\documentclass[11pt,twoside, final]{amsart}
\copyrightinfo{0}{Iranian Mathematical Society}
\usepackage{amsmath,amsthm,amscd,amsfonts,amssymb,enumerate}
\usepackage{graphicx}		
\usepackage{color}
\usepackage{listings}
\usepackage[colorlinks]{hyperref}
\usepackage[numbers,sort&compress]{natbib}

\newtheorem{theorem}{Theorem}[section]
\newtheorem{proposition}[theorem]{Proposition}
\newtheorem{lemma}[theorem]{Lemma}
\newtheorem{corollary}[theorem]{Corollary}
\theoremstyle{definition}

\newtheorem{example}[theorem]{Example}
\theoremstyle{remark}
\newtheorem{remark}[theorem]{Remark}
\numberwithin{equation}{section}

\begin{document}

\title[Finite Wavelet Frames over Prime Fields]{Finite equal norm Parseval Wavelet Frames over Prime Fields}

\author[A. Rahimi and N. Seddighi]{Asghar Rahimi and Niloufar Seddighi$^*$}

\address{Asghar Rahimi, Faculty Mathematics, University of Maragheh, Maragheh, Iran.}
\email{rahimi@maragheh.ac.ir}

\address{Niloufar Seddighi, Faculty Mathematics, University of Maragheh, Maragheh, Iran.}
\email{stu\_seddighi@maragheh.ac.ir}

%%%%%%%%%%%%%%%%%%%%%%%%%%%%%%%%%%%%555

\begin{abstract}
In the framework of wave packet analysis, finite wavelet systems are
particular classes of  finite wave packet systems. In this paper, using
a scaling matrix on a permuted version of the discrete Fourier transform (DFT) of
system generator, we derive a locally-scaled version of the DFT of system genarator
and obtain a finite equal-norm Parseval wavelet frame over prime fields.
We also give a characterization of all multiplicative subgroups
of the cyclic multiplicative group, for which the associated wavelet systems
form frames. Finally, we present some concrete examples as applications of our results.\\

\end{abstract}

\maketitle

\section{\bf Introduction}

The theory of frames  in finite dimensional Hilbert spaces has been recognized through its central role in
signal representation methods \cite{ff, {P.Bal}}. The best known frames in applications including data transmission such as packet communication networks \cite{G.S}, multiple antenna coding \cite{HHS.MAC}, perfect reconstruction filter banks (PRFBs) \cite{KDG.FBS}, quantization error problems in quantized frame expansions \cite{GKV.QF}, and some areas of quantum communication theory \cite{BKP.QU} are finite equal-norm Parseval frames (ENPFs). The potential of these types of frames is due to the high capability of erasure-resilient, speedy implementation of  reconstruction as well as the structure of equal energy frame vectors.
The first comprehensive review of relevant studies on general equal-norm tight frames (ENTFs), and
equal-norm Parseval frames (ENPFs) as the subclasses of ENTFs have been studied in \cite{CK}.
 Some classes of ENPFs such as (general) harmonic frames, Gabor frames (finite discrete Gabor frames) and
 ENPFs with single generator or more generators which have group structures, have been also introduced in \cite{CK}. Several other classes of ENPFs can be found in \cite{CL}.
Moreover, a number of constructive methods to  ENTFs and ENPFs from finite sets of vectors
have been considered \cite{CK, CL}.
Recently,  some convergent constructive methods to ENPFs have been introduced   \cite{BC.AR}.
Nevertheless, there is a lack of variety in classes of structured
ENPFs, obtained from some generating vectors or with a simple structure which are important in the point of low rate of computation and complexity. Furthermore, in practice, only one special class of ENPFs might not be guaranteed to be suitable for all applications.

Traditionally, the classical Gabor transforms and wavelet transforms  have been used to perform
time-frequency (resp. time-scale) analysis of a given function/signal in a Hilbert space, see \cite{arefi.k1,
 arefi.z2, arefi.z1}. In the last decades, generalized methods of Gabor transforms and wavelet transforms
have been developed \cite{AGHF.claswelt}. In the framework of wave packet analysis \cite{Chris.Rahimi, AGHF.thf}, finite wavelet systems are
particular classes of  finite wave packet systems which have been recently introduced, see \cite{AGHF.claswels, {AGHF.wpt.ff}, {AGHF.wpt.n}, {AGHF.cwps.p}}.
Extending the finite wavelet frames over prime fields in \cite{AGHF.SFWFPF},
the analytic structure, group theoretical and abstract aspects  of
the nature of such systems have been studied in \cite{{AGHF.SFWFPF}, {AGHF.MMP}, {RS.AC}}.

In Theorem \ref{parseval}, using a scaling matrix (not necessarily uniform) on a permuted version of the discrete Fourier transform (DFT) of a given window function (system generator) under certain conditions, we
derive a locally-scaled version of the DFT of the window function and obtain
a finite equal-norm Parseval  wavelet frame over prime fields.
In \cite{RS.AC}, we presented a matrix representation of the DFT of the window function, which
was based on a generator of the cyclic multiplicative group of integers modulo  $p$, where $p$ is a prime number.
This notion is a quite useful tool to determine whether a given system
forms a frame.
In this paper, we further apply this matrix notion. For any non-zero window function using this representation,
we give a characterization of all multiplicative subgroups of the cyclic multiplicative group modulo $p$, for which the associated wavelet systems form frames.
Although, this notion depends on a generator of the cyclic multiplicative group (not necessarily a specific generator)
and there is not a general method to find its generator, however this is the key point in cryptosystems.

Construction of ENPFs in such systems relies on a local modifying of the frequency components
of the system generator.
For a prime integer $p$, the multiplicative group modulo $p$
is cyclic and based on that, Proposition \ref{part} provides a convenient correspondence between the multiplicative subgroups and the main group.
This leads us to simply manipulate the DFT of the system generator in order to obtain ENPFs of such systems.

This paper is orgonized as follows.
Section 2 contains some basic definitions and results of Fourier transform on cyclic groups and finite frames.
An overview for the notion and structure of finite wavelet groups
appears in section 3.
In section 4, we give main results of the current paper.
The results will be  accompanied by some concrete examples.

%=============================================================================%
\section{\bf Preliminaries}

Throughout this paper, $p$ is a prime positive integer.
Here, we state a brief review of notations, basics and preliminaries of
Fourier analysis and computational harmonic analysis over finite cyclic groups. For more details, we refer the readers to \cite{AGHF.SFWFPF, AGHF.wpt.ff, AGHF.wpt.n, Pf}.
Here we employ notations of the author in \cite{AGHF.claswels, AGHF.claswelt, AGHF, AGHF.cwps.p, AGHF.SFWFPF, AGHF.wpt.ff, AGHF.wpt.n, AGHF.thf, AGHF.MMP}.

For a finite  group $G$,
\begin{align*}
\mathbb{C}^G=\{\mathbf{x}:G\to\mathbb{C}\}
\end{align*}
is a $|G|$-dimensional complex vector space. For any vector in $\mathbb{C}^G$
the indices are taken to be elements in finite group
$G$.
This space is a Hilbert space under the inner product
\begin{align*}
\langle\mathbf{x},\mathbf{y}\rangle=\sum_{g\in G}\mathbf{x}(g)\overline{\mathbf{y}(g)}
\end{align*}
for any $\mathbf{x},\mathbf{y}\in\mathbb{C}^G$. The induced norm is the $\|.\|_2$-norm.
Let $\mathbb{Z}_p$ denotes the cyclic group of $p$ elements $\{0,...,p-1\}$. Then for $\mathbb{C}^{\mathbb{Z}_p}$,
we write $\mathbb{C}^p$.
Also
\begin{align*}
\|\mathbf{x}\|_0=|\{k\in\mathbb{Z}_p:\mathbf{x}(k)\not=0\}|
\end{align*}
counts non-zero entries in $\mathbf{x}\in\mathbb{C}^p$.

For $k\in \mathbb{Z}_p$, the translation operator $T_k:\mathbb{C}^p\to\mathbb{C}^p$ is defined by
\begin{align*}
T_k\mathbf{x}(s)=\mathbf{x}(s-k), \hspace{1cm}(\mathbf{x}\in\mathbb{C}^p,\hspace{.2cm} s\in\mathbb{Z}_p).
\end{align*}
For $\ell\in \mathbb{Z}_p$, the modulation operator $M_\ell:\mathbb{C}^p\to\mathbb{C}^p$ is defined by
\begin{align*}
M_\ell\mathbf{x}(s)=e^{-2\pi i\ell s/p}\mathbf{x}(s), \hspace{1cm}(\mathbf{x}\in\mathbb{C}^p,\hspace{.2cm} s\in\mathbb{Z}_p).
\end{align*}
These operators are unitary operators in the $\|.\|_2$-norm. For any $\ell,k\in \mathbb{Z}_p$, we have $(T_k)^*=T_{p-k}$ and $(M_{\ell})^*=M_{p-\ell}$.
The unitary discrete Fourier transform (DFT) of  $x\in \mathbb{C}^p$ is defined
by
$\widehat{\mathbf{x}}(\ell)=\frac{1}{\sqrt{p}}\sum_{k=0}^{p-1}\mathbf{x}(k)\overline{\mathbf{w}_{\ell}(k)},$
for all $\ell\in\mathbb{Z}_p$ where for all
$\ell,k\in\mathbb{Z}_p$ we have $\mathbf{w}_{\ell}(k) = e^{2\pi i\ell k/p}.$
Therefore, the DFT of $\mathbf{x}\in\mathbb{C}^p$ at $\ell\in \mathbb{Z}_p$ can be written as
\begin{equation}\label{DFT}
\widehat{\mathbf{x}}(\ell)=\mathcal{F}_p(\mathbf{x})(\ell)
=\frac{1}{\sqrt{p}}\sum_{k=0}^{p-1}\mathbf{x}(k)\overline{\mathbf{w}_{\ell}(k)}
=\frac{1}{\sqrt{p}}\sum_{k=0}^{p-1}\mathbf{x}(k)e^{-2\pi i\ell k/p}.
\end{equation}
For $x\in\mathbb{C}^p$, the inverse discrete Fourier transform(IDFT) is defined by
\begin{equation*}
\mathbf{x}(k)=\frac{1}{\sqrt{p}}\sum_{\ell=0}^{p-1}\widehat{\mathbf{x}}(\ell)\mathbf{w}_{\ell}(k)=
\frac{1}{\sqrt{p}}\sum_{\ell=0}^{p-1}\widehat{\mathbf{x}}(\ell)e^{2\pi i\ell k/p},\ \ 0\le k\le p-1.
\end{equation*}
By $\|.\|_2$-norm, DFT is a unitary transform. Thus, for all $\mathbf{x}\in \mathbb{C}^p$
Parseval formula $\|\mathcal{F}_{p}(\mathbf{x})\|_{2}=\|\mathbf{x}\|_{2}$ satisfies.
The Polarization identity implies $\langle\mathbf{x},\mathbf{y}\rangle
=\langle\widehat{\mathbf{x}},\widehat{\mathbf{y}}\rangle$ for $\mathbf{x},\mathbf{y}\in\mathbb{C}^p$
which is known the Plancherel formula. The unitary DFT (\ref{DFT}) satisfies
\begin{equation*}
\widehat{T_k\mathbf{x}}=M_k\widehat{\mathbf{x}}, \widehat{M_\ell\mathbf{x}}=T_{p-\ell}\widehat{\mathbf{x}}
\hspace{0.5cm}(\mathbf{x}\in\mathbb{C}^p,\hspace{.2cm}(k,\ell\in\mathbb{Z}_p).
\end{equation*}
A finite sequence $\mathfrak{A}=\{\mathbf{y}_j:0\le j\le M-1\}\subset\mathbb{C}^p$
is called a frame (or finite frame) for $\mathbb{C}^p$,
if there exists a positive constant $0<A\leq B<\infty$ such that
\begin{equation}\label{f.c}
A\|\mathbf{x}\|_{2}^2\le \sum_{j=0}^{M-1}|\langle
\mathbf{x},\mathbf{y}_j\rangle|^2\le B\|\mathbf{x}\|_{2}^2,\  \  \ (\mathbf{x}\in\mathbb{C}^{p}).
\end{equation}
If (\ref{f.c}) satisfies only the upper bound then $\mathfrak{A}$ is a Bessel sequence. Any finite sequence in
$\mathbb{C}^p$ is a Bessel sequence, so that the condition in (\ref{f.c})
can be reduced to
\begin{equation*}
A\|\mathbf{x}\|_{2}^2\le \sum_{j=0}^{M-1}|\langle
\mathbf{x},\mathbf{y}_j\rangle|^2,\  \  \ (\mathbf{x}\in\mathbb{C}^{p}).
\end{equation*}
for $0<A<\infty$. If $A=B$, then $\mathfrak{A}$ is a $A-$tight frame and if $A=B=1$, it is called a Parseval frame.
Moreover, if all the frame elements have the same norm it is called equal-norm frame and if all the elements have
norm 1 it is called unit-norm frame.

If $\mathfrak{A}=\{\mathbf{y}_j:0\le j\le M-1\}$ is a frame for $\mathbb{C}^p$, the synthesis operator
$F:\mathbb{C}^M\to\mathbb{C}^p$ is defined by
\begin{align*}
F\{c_j\}_{j=0}^{M-1}=\sum_{j=0}^{M-1}c_j\mathbf{y}_j, \  \  \ (\{c_j\}_{j=0}^{M-1}\in\mathbb{C}^M).
\end{align*}
The adjoint operator $F^*:\mathbb{C}^p\to\mathbb{C}^M$ which is known as analysis operator is defined by
\begin{align*}
F^*\mathbf{x}=\{\langle\mathbf{x},\mathbf{y}_j\rangle\}_{j=0}^{M-1},\  \  \ \ (\mathbf{x}\in\mathbb{C}^p).
\end{align*}
The frame operator $S:\mathbb{C}^p\to\mathbb{C}^p$ is defined by
\begin{equation*}\label{S}
\mathbf{x}\mapsto S\mathbf{x}=FF^*\mathbf{x}
=\sum_{j=0}^{M-1}\langle\mathbf{x},\mathbf{y}_j\rangle\mathbf{y}_j,\  \  \  \ (\mathbf{x}\in\mathbb{C}^{p}),
\end{equation*}
and $\mathbf{x}=\sum_{j=0}^{M-1}\langle\mathbf{x},S^{-1}\mathbf{y}_j\rangle\mathbf{y}_j
=\sum_{j=0}^{M-1}\langle\mathbf{x},\mathbf{y}_j\rangle S^{-1}\mathbf{y}_j.$
The redundancy of the finite frame $\mathfrak{A}$ is defined  by
 ${\rm red}_{\mathfrak{A}}=M/p$ where $M=|\mathfrak{A}|$.
%%%%%%%%%%%%%%%%%%%%%%%%%%%%%%%%%%%%%%%%%%%%%%%%%%%%%%%

\section{\bf Construction of Wavelet Frames over Prime Fields}

In this section, we briefly state structure and basic properties of
cyclic dilation operators (cf. \cite{f.g.h.t, AGHF.SFWFPF, AGHF, j}). We shall present an overview for the notion and structure of finite wavelet groups over prime fields.
%%%%%%%%%%%%%%%%%%%%%%%%%%%%%%%%

\subsection{\bf Structure of Finite Wavelet Group over Prime Fields}

The set
\begin{equation*}
\mathbb{U}_p:=\mathbb{Z}_p-\{0\}=\{1,...,p-1\},
\end{equation*}
is a finite multiplicative Abelian group of order $p-1$
with respect to the multiplication module $p$ with the identity element ${1}$.
The multiplicative inverse for $m\in\mathbb{U}_p$
is $m_p$  satisfying $m_pm+np=1$ for some $n\in\mathbb{Z}$ (cf.  \cite{Hardy}).

For $m\in\mathbb{U}_p$,  the cyclic dilation operator
$D_m:\mathbb{C}^p\to\mathbb{C}^p$  is defined by
\begin{align*}
D _m\mathbf{x}(k):=\mathbf{x}(m_pk)
\end{align*}
for $\mathbf{x}\in\mathbb{C}^p$ and $k\in\mathbb{Z}_p$,
where $m_p$ is the multiplicative inverse of $m$ in $\mathbb{U}_p$.

In the following propositions, we state some basic properties of the cyclic dilations.

\begin{proposition}\label{prop30}
{\it Let $p$ be a positive prime integer. Then
\begin{enumerate}[\normalfont (i)]
\item For $(m,k)\in\mathbb{U}_p\times\mathbb{Z}_p$, we have $D_mT_k=T_{mk}D_m$.
\item For $m,m'\in\mathbb{U}_p$, we have $D_{mm'}=D_mD_{m'}$.
\item For $(m,k),(m',k')\in\mathbb{U}_p\times\mathbb{Z}_p$, we have
$T_{k+mk'}D_{mm'}=T_kD_mT_{k'}D_{m'}$.
\item For $(m,\ell)\in \mathbb{U}_p\times\mathbb{Z}_p$, we have $D_mM_\ell=M_{m_p\ell}D_m$.
\end{enumerate}
}\end{proposition}
%------------------------------------------------------------------------------------------%
The next result also  states some properties of the cyclic dilations.

\begin{proposition}\label{prop31}
{\it Let $p$ be a positive prime integer and $m\in\mathbb{U}_p$.
Then
\begin{enumerate}[\normalfont (i)]
\item The dilation operator $D_m:\mathbb{C}^p\to\mathbb{C}^p$ is a $*$-homomorphism.
\item The dilation operator $D_m:\mathbb{C}^p\to\mathbb{C}^p$ is
unitary in $\|.\|_2$-norm and satisfies
\[
(D_m)^*=(D_m)^{-1}=D_{m_p}.
\]
\item For $\mathbf{x}\in\mathbb{C}^p$, we have
$\widehat{D_m\mathbf{x}}=D_{m_p}\widehat{\mathbf{x}}$.
\end{enumerate}
}\end{proposition}
%----------------------------------------------------------------------------------------------%
The underlying set
$$\mathbb{U}_p\times\mathbb{Z}_p=\left\{(m,k):m\in\mathbb{U}_p, k\in\mathbb{Z}_p\right\},$$
equipped with the following group operations
\begin{equation*}
(m,k)\ltimes (m',k'):=(mm',k+mk'),
\end{equation*}
\begin{equation*}
(m,k)^{-1}:=(m_p,m_p.(p-k)),
\end{equation*}
denoted by $\mathbb{W}_p$, is a finite non-Abelian group of order $p(p-1)$  and it is called as
finite affine group on $p$ integers or
 the finite wavelet group associated to the integer $p$ or simply as $p$-wavelet group.

Next proposition states basic properties of the finite wavelet group $\mathbb{W}_p$.

\begin{proposition}
Let $p>2$ be a positive prime integer. Then $\mathbb{W}_p$ is a non-Abelian group of order $p(p-1)$ which contains
 a copy of $\mathbb{Z}_p$ as a normal Abelian subgroup and a copy of $\mathbb{U}_p$
 as a non-normal Abelian subgroup.
\end{proposition}

\subsection{\bf Wavelet Frames over Prime Fields}

A finite wavelet system for the complex Hilbert space $\mathbb{C}^p$ is
a family or system of the form
\begin{equation*}
\mathcal{W}({\mathbf{y}},\Delta)
:=\left\{\sigma(m,k)\mathbf{y}=T_kD_m\mathbf{y}:(m,k)\in\Delta\subseteq\mathbb{W}_p\right\},
\end{equation*}
for some window signal $\mathbf{y}\in\mathbb{C}^p$ and a subset $\Delta$ of $\mathbb{W}_p$.

If $\Delta=\mathbb{W}_p$ we put $\mathcal{W}(\mathbf{y}):=\mathcal{W}(\mathbf{y},\mathbb{W}_p)$
and it is called a full finite wavelet system over $\mathbb{Z}_p$.
A finite wavelet system which spans $\mathbb{C}^p$ is a frame
and is referred to as a finite wavelet frame over the prime field $\mathbb{Z}_p$.

If $\mathbf{y}\in\mathbb{C}^p$ is a window signal then for $\mathbf{x}\in\mathbb{C}^p$, we have
\begin{equation*}\label{w.r.i}
\langle\mathbf{x},\sigma(m,k){\mathbf{y}}\rangle
=\langle\mathbf{x},T_kD_m{\mathbf{y}}\rangle
=\langle T_{p-k}\mathbf{x},D_m{\mathbf{y}}\rangle,\ \ \  \ ((m,k)\in\mathbb{W}_p).
\end{equation*}

The following proposition states a formulation for wavelet coefficients via Fourier transform.

\begin{proposition}\label{w.f.rep}
{\it Let $\mathbf{x},\mathbf{y}\in\mathbb{C}^p$ and $(m,k)\in\mathbb{W}_{p}$. Then,
$$\langle\mathbf{x},\sigma(m,k){\mathbf{y}}\rangle
=\sqrt{p}\mathcal{F}_p(\widehat{\mathbf{x}}.\overline{\widehat{D_m{\mathbf{y}}}})(p-k).$$
}\end{proposition}
\begin{proof}
See Proposition 4.1 of \cite{AGHF.SFWFPF}.
\end{proof}

%%%%%%%%%%%%%%%%%%%%%%%%%%%%%%%%%%%%%%%%%%%%%%%%%
Using an analytic approach, the author \cite{AGHF.SFWFPF} has presented a
concrete formulation for the $\|.\|_2$-norm of wavelet coefficients
the formula of which is just stated hereby.

\begin{theorem}\label{f.f.s}
Let $p$ be a positive prime integer, $M$  be a divisor of $p-1$, and let
$\mathbb{M}$ be a multiplicative subgroup of $\mathbb{U}_p$ of order $M$.
Let $\mathbf{y}\in\mathbb{C}^p$ be a window signal and $\mathbf{x}\in\mathbb{C}^p$. Then,
\begin{align*}
&\sum_{m\in\mathbb{M}}\sum_{k\in\mathbb{Z}_p}|\langle\mathbf{x},\sigma(m,k)\mathbf{y}\rangle|^2
\\&=p\left(M|\widehat{{\mathbf{y}}}(0)|^2|\widehat{\mathbf{x}}(0)|^2
+\left(\sum_{m\in\mathbb{M}}|\widehat{{\mathbf{y}}}(m)|^2\right)\left(\sum_{\ell\in\mathbb{M}}|\widehat{\mathbf{x}}(\ell)|^2\right)
+\sum_{\ell\in\mathbb{U}_p-\mathbb{M}}\gamma_\ell(\mathbf{y},\mathbb{M})|\widehat{\mathbf{x}}(\ell)|^2\right),
\end{align*}
where
\[
\gamma_\ell(\mathbf{y},\mathbb{M}):=\sum_{m\in\mathbb{M}}|\widehat{\mathbf{y}}(m\ell)|^2,\ \ \ \ \ (\ell\in\mathbb{U}_p-\mathbb{M}).
\]
\end{theorem}

\begin{proof}
 See Theorem 4.2 of \cite{AGHF.SFWFPF}.
\end{proof}

\begin{remark}
The formulation presented in Theorem \ref{f.f.s} is an analytic
formulation of wavelet coefficients associated to the subgroup $\mathbb{M}$.
In detail, that formulation originated from an analytic
approach which was based on direct computations of cyclic dilations in the subgroup $\mathbb{M}$.
\end{remark}
%%%%%%%%%%%%%%%%%%%%%%%%%%%%%%%%%%%%%%%%%%%%%%
Next proposition provides a particular partition of the cyclic multiplicative group
$\mathbb{U}_p$. Applying this, a constructive formulation for the $\|.\|_2$-norm of wavelet coefficients
has been achieved in the following theorem.
\begin{proposition}\label{part}
{\it Let $p$ be a positive prime integer, $M$ be a divisor of $p-1$,
and let $\mathbb{M}$ be a multiplicative subgroup of $\mathbb{U}_p$ of order $M$.
Let $\epsilon$ be a generator of the cyclic group $\mathbb{U}_p$ and
$a:=\frac{p-1}{M}$.
Then
\begin{enumerate}[\normalfont (i)]
\item For $0\le r,s\le a-1$, we have $\epsilon^r\mathbb{M}=\epsilon^s\mathbb{M}$ iff $r=s$.
\item $\mathbb{U}_p/\mathbb{M}=\{\epsilon^t\mathbb{M}:0\le t\le a-1\}$.
\end{enumerate}
}
\end{proposition}
%__________________________________________________________________________
\begin{proof}
See Proposition 3.7 of \cite{RS.AC}.
\end{proof}
%%%%%%%%%%%%%%%%%%%%%%%%%%%%%%%%%%%%%%%%%%%%%%%%%%
The following theorem presents a constructive formultion for the $\|.\|_2-$norm of wavelet coeficients.
%%%%%%%%%%%%%%%%%%%%%%%%%%%%%%%%%%%%%%%%%%%%%%%%%
\begin{theorem}\label{f.f.s.coset}
Let $p$ be a positive prime integer, $M$  be a divisor of $p-1$,
and let $\mathbb{M}$ be a multiplicative subgroup of $\mathbb{U}_p$ of order $M$.
Let $\epsilon$ be a generator of the cyclic group $\mathbb{U}_p$ and
$a:=\frac{p-1}{M}$.
Let $\mathbf{y}\in\mathbb{C}^p$ be a window signal and
$\mathbf{x}\in\mathbb{C}^p$.
Then,
\begin{align*}\label{expan.coset}
&\sum_{m\in\mathbb{M}}\sum_{k=0}^{p-1}|\langle\mathbf{x},T_kD_m\mathbf{y}\rangle|^2
\\&\hspace{1cm}=p\bigg(M{|\widehat{\mathbf{x}}(0)|}^{2} {|\widehat{\mathbf{y}}(0)|}^{2}
+ \sum_{t=0}^{a-1}\bigg(\sum_{\ell\in H_t}{|\widehat{\mathbf{x}}(\ell)|}^{2}\bigg)
\bigg(\sum_{w\in H_t}{|\widehat{\mathbf{y}}(w)|}^{2}\bigg)\bigg),
\end{align*}
where $H_t:=\epsilon^t \mathbb{M}$ for all $0\le t\le a-1$.
\end{theorem}
%___________________________________________________________________________________
\begin{proof}
See Theorem 3.8 \cite{RS.AC}.
\end{proof}
%%%%%%%%%%%%%%%%%%%%%%%%%%%%%%%%%%%%%%%%%%%%%%%%%%%%%%%%%%%
For a given multiplicative subgroup of $\mathbb{U}_p$, the next theorem
gives necessary and sufficient conditions for a finite wavelet system over prime field to be a frame.
\begin{theorem}\label{m.f.result}
{\it Let $p$ be a positive prime integer, $\epsilon$ be a generator of
$\mathbb{U}_p$,
 $M$ be a divisor of $p-1$,
$\mathbb{M}$ be a multiplicative subgroup of $\mathbb{U}_p$ of order $M$,  and let
$a:=\frac{p-1}{M}$.
Let $\Delta_\mathbb{M}:=\mathbb{M}\times\mathbb{Z}_p$ and $\mathbf{y}\in\mathbb{C}^p$ be a non-zero window signal.
The finite wavelet system $\mathcal{W}(\mathbf{y},\Delta_\mathbb{M})$ is a frame for
$\mathbb{C}^p$ if and only if the following conditions hold
\begin{enumerate}[\normalfont (i)]
\item $\widehat{\mathbf{y}}(0)\neq 0$
\item For each $t\in \{0,...,a-1\}$, there exists $m_t\in \mathbb{M}$ such that  $\widehat{\mathbf{y}}(\epsilon^t m_t )\neq 0$.
\end{enumerate}
}
\end{theorem}
%_______________________________________________________________________________________
\begin{proof}
See Theorem 3.9 \cite{RS.AC}
\end{proof}
%%%%%%%%%%%%%%%%%%%%%%%%%%%%%%%%%%%%%%%%%%%%%%%%%%
The next result, in the matrix language also gives a constructive characterization for the frame conditions
of finite wavelet systems over prime fields.

\begin{corollary}\label{matrix.int}
{\it Let $p$ be a positive prime integer, $\epsilon$ be a generator of
$\mathbb{U}_p$,
 $M$ be a divisor of $p-1$,
$\mathbb{M}$ be a multiplicative subgroup of $\mathbb{U}_p$ of order $M$, and let
$a:=\frac{p-1}{M}$.
Let $\Delta_\mathbb{M}:=\mathbb{M}\times\mathbb{Z}_p$ and $\mathbf{y}\in\mathbb{C}^p$ be a non-zero window signal.
The finite wavelet system $\mathcal{W}(\mathbf{y},\Delta_\mathbb{M})$ is a frame for
$\mathbb{C}^p$ if and only if  $\widehat{\mathbf{y}}(0)\neq 0$ and the matrix ${\mathbf{Y}}({\mathbb{M},\mathbf{y}})$ of size $a\times M$ given by
$$
{\mathbf{Y}}({\mathbb{M},\mathbf{y}}):=\begin{bmatrix}
 \widehat{\mathbf {y}}(1)& \widehat{\mathbf {y}}(\epsilon^{a})& \cdots & \widehat{\mathbf {y}}(\epsilon^{(M-1)a}) \\
 \widehat{\mathbf {y}}(\epsilon^{1})& \widehat{\mathbf {y}}(\epsilon^{a+1})& \cdots & \widehat{\mathbf {y}}(\epsilon^{(M-1)a+1}) \\
 \vdots & \vdots & \ddots & \vdots\\
  \widehat{\mathbf {y}}(\epsilon^{a-1})& \widehat{\mathbf {y}}(\epsilon^{2a-1})&  \cdots&\widehat{\mathbf {y}}(\epsilon^{Ma-1})
 \end{bmatrix}_{a\times M}
 $$
is a matrix such that each row has at least a non-zero entry.}
\end{corollary}
%______________________________________________________________________________________________________
\begin{proof}
See Corollary 3.10 \cite{RS.AC}.
\end{proof}
%_________________________________________________________________________________________________

\section{\bf Finite equal norm Parseval wavelet frames over prime fields}

{
Weight frames, controlled frames \cite{B.WC}, and scalable frames \cite{KO.SC} have been already applied in order to tighten
and
also to produce Parseval frame of a given frame.
However, in general
the motivation of these methods in \cite{B.WC, KO.SC} have not been aimed at the norm equality of the frame vectors.
These methods are applied not only in finite dimensional Hilbert spaces but also in infinite dimensions.
In fact, in  the case of infinite dimensions, weighted frames algorithms are designed in such a way that decrease
the condition number and provide approximately a tight frame.
In \cite{KO.SC}, some equivalent conditions to scalable frames have been introduced.
The principal significance of the constructive approach in  the next theorem is that we apply
a local scaling on the DFT of a given window function. The structure of these
systems allows one
to simply manipulate  the DFT  of the system generator, in order to obtain an ENPF.

The following lemma provides an algebraic tool.
This permutation depends on
a multiplicative subgroup
$\mathbb{M}$
of $\mathbb{U}_p$ of a given size. Note that,
for each divisor of the size of a
finite cyclic group, there is exactly one subgroup of that size.
%%%%%%%%%%%%%%%%%%%%%%%%%%%%%%%%%%%%%%%%%%%%%%%%%%
\begin{lemma}\label{permutation}
 Let $p$ be a positive prime integer, $\epsilon$ be a generator of $\mathbb{U}_p$,  $M$  be a divisor of $p-1$,
and let $a:=\frac{p-1}{M}$. Then
\begin{equation}\label{sigma2}
\sigma(\ell):=\begin{cases}
0,& \ell=0,\\
\epsilon^{\lfloor{\frac{\ell-1}{M}}\rfloor+
(\ell - \lfloor{\frac{\ell-1}{M}}\rfloor M -1)a},& \ell\in \mathbb{U}_p,
\end{cases}
\end{equation}
is a permutation of $\mathbb{Z}_p$ where $\lfloor.\rfloor$ denotes the floor function.
\end{lemma}
%_______________________________________________________________________________________
\begin{proof}
Clearly $\sigma$ is well-defined. In order to show that $\sigma$ is bijective it suffices to prove that $\sigma:\mathbb{U}_p\longrightarrow \mathbb{U}_p$ is surjective.  For any $\ell\in\mathbb{U}_p$ there exists $0\leq t\leq a-1$ such that $tM+1\leq \ell\leq (t+1)M$ and so it is clear that $\lfloor \frac{\ell-1}{M}\rfloor=t$. Thus the definition implies that $\sigma(\ell)=\epsilon^{t+
(\ell-t M -1)a}$; in particular if $1\leq \ell\leq M$, we have $\sigma(\ell)=\epsilon^{(\ell-1)a}$. We observe that $\mathbb{M}=<\epsilon^a>$ is a subgroup of $\mathbb{U}_p$ of order $M$ and  $\mathbb{U}_p=\bigcup_{t=0}^{a-1}\limits\epsilon^t \mathbb{M}$, where
$\epsilon^t \mathbb{M}=\{\epsilon^t,...,\epsilon^{t+(M-1)a}\}$ for any $0\leq t\leq a-1$ and $\epsilon^t \mathbb{M}\cap\epsilon^s \mathbb{M}=\emptyset$ for any $t\neq s$. Given an arbitrary $x\in\mathbb{U}_p$, there exists $0\leq t\leq a-1$ such that $x\in\epsilon^t \mathbb{M}$ and so there exists $0\leq k\leq M-1$ such that $x=\epsilon^{t+ka}$. One can easily check that $\sigma(tM+k+1)=x$.
\end{proof}
%%%%%%%%%%%%%%%%%%%%%%%%%%%%%%%%%%%%%%%%%%%%%%%%%%%%%%%%%%%

Here, we apply lemma \ref{permutation}, in the following theorem.
At first, the permutation defined in (\ref{sigma2}), is
served
 to give
a regular
rearrangement  to frequency components of the DFT of
the
 window function
and classify them in order to
set cluster scales.
Next again via this permutation, we
are able to reverse the locally-scaled permuted version of the DFT of the window function to
merely a
locally-scaled version of the window function. Since the scales are defined positive
thus, still this last version
 fulfills the frame conditions.
Under the described procedure, we are
thus led
to construct a finite equal-norm Parseval
frame of such systems for the Hilbert space $\mathbb{C}^p$.

%%%%%%%%%%%%%%%%%%%%%%%%%%%%%%%%%%%%%%%%%%%%%%%%%%%%%%%%%%%%%%%
\begin{theorem}\label{parseval}
{\it  Let $p$ be a positive prime integer, $\epsilon$ be
  a generator of $\mathbb{U}_p$
, $\mathbb{M}$ be a multiplicative subgroup of $\mathbb{U}_p$
of order $M$,
 $a:= \frac{p-1}{M}$ be the index of $\mathbb{M}$ in $\mathbb{U}_p$,
 and let $\sigma$ is defined by
\begin{align}\label{sigma}
\sigma:=\begin{pmatrix}
0 & 1 & 2 &\cdots & \ell & \cdots & p-1 \\
0 & \epsilon^0 & \epsilon^a & \cdots & \epsilon^{\lfloor{\frac{\ell-1}{M}}\rfloor+
(\ell - \lfloor{\frac{\ell-1}{M}} \rfloor M -1)a} & \cdots & \epsilon^{(Ma -1)}
\end{pmatrix}.
\end{align}
Let $\mathbf{y}\in\mathbb{C}^p$ be a non-zero window signal and
\begin{align}\label{yzeg}
\nonumber &{\widehat{\mathbf{y}}}'(\ell):= \widehat{\mathbf{y}}(\sigma(\ell))=
\begin{cases}
\widehat{\mathbf{y}}(0),& \ell=0,\\
\widehat{\mathbf{y}}(\epsilon^{\lfloor{\frac{\ell-1}{M}}\rfloor+
(\ell - \lfloor{\frac{\ell-1}{M}} \rfloor M -1)a}),& \ell\in \mathbb{U}_p,
\end{cases}\\&
{\widehat{\mathbf{y}}}'':= D(\mathbb{M},\mathbf{y})\widehat{\mathbf{y}}',
\end{align}
where
$$
{\mathbf{D}}({\mathbb{M},\mathbf{y}}):=
\begin{bmatrix}
\begin{bmatrix}
U'
\end{bmatrix}_{1\times 1}
\\
&\begin{bmatrix}
U_0
\end{bmatrix}_{M\times M}\\
&&\begin{bmatrix}
U_1
\end{bmatrix}_{M\times M}\\
&&&\ddots\\
&&&&\begin{bmatrix}
U_{a-1}
\end{bmatrix}_{M\times M}\\
\end{bmatrix}_{p\times p}
$$
is a diagonal matrix with diagonal block matrices
$\begin{bmatrix}
U'
\end{bmatrix}_{1\times 1}$
and
$\begin{bmatrix}
U_t
\end{bmatrix}_{M\times M}$
for $ t\in\{0,...,a-1\}$,
which are constant along diagonals by entries
$R':= \frac{1}{\sqrt{pM}|\widehat{\mathbf{y}}(0)|^2}$
and
$R_t:=\frac{1}{\sqrt{\sum_{r_t=0}^{M-1} p |\widehat{\mathbf{y}}(\epsilon^{r_t a+t})|^2}}$,
respectively.
 If $\mathbf{y}$ satisfies the following conditions
\begin{enumerate}[\normalfont (i)]
\item $\widehat{\mathbf{y}}(0)\neq 0$,
\item for each $t\in \{0,...,a-1\}$, there exists $r_t\in \{0,...,M-1\}$ such that  $\widehat{\mathbf{y}}(\epsilon^{r_t  a+t})\neq 0$,
\end{enumerate}
then $\mathcal{W}(\mathbf{y_{\sigma}},\Delta_\mathbb{M})$ is an equal-norm Parseval frame for $\mathbb{C}^p$, where
$\widehat{\mathbf{y}}_{\sigma}(\sigma(\ell)):={\widehat{\mathbf{y}}}''(\ell)$.
}
\end{theorem}
%_______________________________________________________________________________________
\begin{proof}
Let $\mathbf{y}\in \mathbb{C}^p$ be a window function which satisfies conditions (i) and (ii)
and  $\mathbf{x}\in \mathbb{C}^p$. Then
$|\widehat{\mathbf{y}}(0)|^2\neq 0$ and for any $t\in \{0,...,a-1\}$, we have
$ \sum_{r_t=0}^{M-1} |\widehat{\mathbf{y}}(\epsilon^{r_t a+t})|^2\neq0.$
Thus, $R'$ and $R_t$ are well-defined for all $t\in \{0,...,a-1\}$.
By Lemma \ref{permutation}, $\sigma$ presents a permutation of $\mathbb{Z}_p$ and
$\sigma^{-1}(\epsilon^{t+r_ta})=r_t+tM+1$,
for any $t\in \{0,...,a-1\}$ and $r_t\in \{0,...,M-1\}.$
By Theorem \ref{f.f.s.coset}, we obtain
\begin{align*}
&\sum_{m\in\mathbb{M}}\sum_{k=0}^{p-1}|\langle\mathbf{x},T_kD_m\mathbf{y_{\sigma}}\rangle|^2
=p M|\widehat{\mathbf{x}}(0)|^2|{\widehat{\mathbf{y}}_{\sigma}}(0)|^2+ p\sum_{t=0}^{a-1}\bigg(\sum_{\ell\in H_t}|
\widehat{\mathbf{x}}(\ell)|^2\bigg)\bigg(\sum_{w\in H_t}|{\widehat{\mathbf{y}}_{\sigma}}(w)|^2\bigg)
\\&=p M|\widehat{\mathbf{x}}(0)|^2|{\widehat{\mathbf{y}}_{\sigma}}(0)|^2+ p\sum_{t=0}^{a-1}\bigg(\sum_{\ell\in H_t}|
\widehat{\mathbf{x}}(\ell)|^2\bigg)\bigg(\sum_{r_t=0}^{M-1}\bigg
|{\widehat{\mathbf{y}}_{\sigma}}({\epsilon}^{r_t a+t})\bigg|^2\bigg)
\\& = p M|\widehat{\mathbf{x}}(0)|^2|\widehat{\mathbf{y}}''(0)|^2+ p\sum_{t=0}^{a-1}\bigg(\sum_{\ell\in H_t}|
\widehat{\mathbf{x}}(\ell)|^2\bigg)\bigg(\sum_{r_t=0}^{M-1}\bigg|\widehat{\mathbf{y}}''\bigg({\sigma}^{-1}
({\epsilon}^{r_ta+t})\bigg)\bigg|^2\bigg)
\\& = p M|\widehat{\mathbf{x}}(0)|^2|R'\widehat{\mathbf{y}}(0)|^2+ p\sum_{t=0}^{a-1}\bigg(\sum_{\ell\in H_t}|
\widehat{\mathbf{x}}(\ell)|^2\bigg)\bigg(\sum_{r_t=0}^{M-1}\bigg|\widehat{\mathbf{y}}''\bigg({\sigma}^{-1}
({\epsilon}^{r_ta+t})\bigg)\bigg|^2\bigg)
\\& = |\widehat{\mathbf{x}}(0)|^2 + p\sum_{t=0}^{a-1}\bigg(\sum_{\ell\in H_t}|
\widehat{\mathbf{x}}(\ell)|^2\bigg)\bigg(\sum_{r_t=0}^{M-1}\bigg|\widehat{\mathbf{y}}''(tM+r_t+1)\bigg|^2\bigg)
\\& = |\widehat{\mathbf{x}}(0)|^2+ p\sum_{t=0}^{a-1}\bigg(\sum_{\ell\in H_t}|
\widehat{\mathbf{x}}(\ell)|^2\bigg)\bigg(\sum_{r_t=0}^{M-1}\bigg|(D(\mathbb{M},\mathbf{y})\widehat{\mathbf{y}}')
(tM+r_t+1)\bigg|^2\bigg)
\end{align*}
Now by (\ref{yzeg}), we get
\begin{align*}
{\widehat{\mathbf{y}}}''(\ell):= D(\mathbb{M},\mathbf{y})\widehat{\mathbf{y}}'(\ell)=
\begin{cases}
R'\widehat{\mathbf{y}}'(0),& \ell=0,\\
R_{\lfloor{\frac{\ell-1}{M}}\rfloor}\widehat{\mathbf{y}}'(\ell),& \ell\in \mathbb{U}_p.
\end{cases}
\end{align*}
Thus, ${\widehat{\mathbf{y}}}''(tM+r_t+1)= R_t \widehat{\mathbf{y}}'(tM+r_t+1)$ and we have
\begin{align*}
&\sum_{m\in\mathbb{M}}\sum_{k=0}^{p-1}|\langle\mathbf{x},T_kD_m\mathbf{y_{\sigma}}\rangle|^2
\\& = |\widehat{\mathbf{x}}(0)|^2
+ p\sum_{t=0}^{a-1}\bigg(\sum_{\ell\in H_t}|
\widehat{\mathbf{x}}(\ell)|^2\bigg)\bigg(\sum_{r_t=0}^{M-1}\bigg|R_t \widehat{\mathbf{y}}'(tM+r_t+1)\bigg)\bigg|^2\bigg)
\\& = |\widehat{\mathbf{x}}(0)|^2
+ p\sum_{t=0}^{a-1}\bigg(\sum_{\ell\in H_t}|
\widehat{\mathbf{x}}(\ell)|^2\bigg)\bigg(\sum_{r_t=0}^{M-1}\bigg|R_t\widehat{\mathbf{y}}\bigg(\sigma(tM+r_t+1)\bigg)\bigg|^2\bigg)
\\& = |\widehat{\mathbf{x}}(0)|^2
+ p\sum_{t=0}^{a-1}\bigg(\sum_{\ell\in H_t}|
\widehat{\mathbf{x}}(\ell)|^2\bigg)\bigg(\sum_{r_t=0}^{M-1}\bigg|R_t\widehat{\mathbf{y}}(\epsilon^{r_ta+t})\bigg|^2\bigg)
\\& = |\widehat{\mathbf{x}}(0)|^2
+ p\sum_{t=0}^{a-1}\bigg(\sum_{\ell\in H_t}|
\widehat{\mathbf{x}}(\ell)|^2\bigg)\bigg(\sum_{r_t=0}^{M-1}\bigg|\frac{1}{\sqrt{\sum_{r_t=0}^{M-1} p |\widehat{\mathbf{y}}(\epsilon^{r_ta+t})|^2}}\widehat{\mathbf{y}}(\epsilon^{r_ta+t})\bigg|^2\bigg)
\\& = |\widehat{\mathbf{x}}(0)|^2 +  \sum_{t=0}^{a-1}\bigg(\sum_{\ell\in H_t}|
\widehat{\mathbf{x}}(\ell)|^2\bigg)
= |\widehat{\mathbf{x}}(0)|^2 + \sum_{\ell=1}^{p-1}
|\widehat{\mathbf{x}}(\ell)|^2 =  \|\widehat{\mathbf{x}}\|_{2}^2=  \|\mathbf{x}\|_{2}^2.
\end{align*}
Therefore, $\mathcal{W}(\mathbf{y_{\sigma}},\Delta_\mathbb{M})$ is a Parseval frame for $\mathbb{C}^p$.
Note that, for any $k\in \mathbb{Z}_p$ and $m\in \mathbb{M}$, the operators $T_k$ and $D_m$ are unitary
operators and,
\begin{align*}
\| T_k D_m \mathbf{y_{\sigma}}\|_{2}=
\|\mathbf{y_{\sigma}}\|_{2}.
\end{align*}
Hence, $\mathcal{W}(\mathbf{y_{\sigma}},\Delta_\mathbb{M})$ is an equal-norm Parseval frame for
$\mathbb{C}^p$.
\end{proof}
%%%%%%%%%%%%%%%%%%%%%%%%%%%%%%%%%%%%%%%%%%%%%%%%%%%%%%%%%%%%
Next example demonstrates the design method of equal-norm Parseval finite wavelet system over prime fields, described
in Theorem \ref{parseval}, in a numerical case.
\begin{example}
Let $p=7$ and $\mathbb{U}_7$ be the multiplicative group modulo 7.
Then $|\mathbb{U}_p|=6$.
One of its generators is 3.
Consider $\mathbb{M}$ be a subgroup of size 3 of $\mathbb{U}_7$.
So
\begin{align*}
\mathbb{M}=<3^{\frac{6}{(3,6)}}>=<2>=
\{1,2,4\}.
\end{align*}
We use Theorem \ref{parseval}
to choose an appropriate window signal $\mathbf{y}\in \mathbb{C}^7$
which satisfies the conditions (i) and (ii).

For simplicity in computation, we will consider a special case with least number
of non-zero components
of $\widehat{\mathbf{y}}$.
Since the conditions (i) and (ii) rely on the DFT of $\mathbf{y}$,
so we use IDFT
to get $\mathbf{y}$.
Now, let $$\widehat{\mathbf{y}}=(1,1,0,1,0,0,0).$$ Then
\begin{align*}
&\mathbf{y}=( 1.1339, 0.2731 + 0.4595i,
  0.5295 + 0.0730i, -0.0467 + 0.5325i,
\\& -0.0467 - 0.5325i, 0.5295 - 0.0730i
, 0.2731 - 0.4595i ),
 \end{align*}
and $\mathcal{W}(\mathbf{y},\Delta_\mathbb{M})$ is a frame for $\mathbb{C}^7$.
By Theorem \ref{parseval}, we have
\begin{align*}
\sigma:=\begin{pmatrix}
0 & 1 & 2 & 3 & 4 & 5 & 6 \\
0 & 3^0 & 3^2 & 3^4 & 3^1 & 3^3 & 3^5
\end{pmatrix}
=\begin{pmatrix}
0 & 1 & 2 & 3 & 4 & 5 & 6 \\
0 & 1 & 2 & 4 & 3 & 6 & 5
\end{pmatrix},
\end{align*}
and
\begin{align*}
&\widehat{\mathbf{y}}'=(\widehat{\mathbf{y}}({\sigma}(0)), \widehat{\mathbf{y}}({\sigma}(1)),\widehat{\mathbf{y}}({\sigma}(2)),\widehat{\mathbf{y}}({\sigma}(3))
,\widehat{\mathbf{y}}({\sigma}(4)),\widehat{\mathbf{y}}({\sigma}(5)),\widehat{\mathbf{y}}({\sigma}(6))
)\\&=(\widehat{\mathbf{y}}(0), \widehat{\mathbf{y}}(1), \widehat{\mathbf{y}}(2), \widehat{\mathbf{y}}(4), \widehat{\mathbf{y}}(3),
\widehat{\mathbf{y}}(6), \widehat{\mathbf{y}}(5))=(1,1,0,0,1,0,0).
\end{align*}
 Also
$$
{\mathbf{D}}({\mathbb{M},\mathbf{y}}):=
\begin{bmatrix}
\begin{bmatrix}
\frac{1}{\sqrt{21}}
\end{bmatrix}_{1\times 1}
\\
&\begin{bmatrix}
\frac{1}{\sqrt{7}}\\
&&\frac{1}{\sqrt{7}}\\
&&& \frac{1}{\sqrt{7}}\\
\end{bmatrix}_{3\times 3}\\
&&\begin{bmatrix}
\frac{1}{\sqrt{7}}\\
&&\frac{1}{\sqrt{7}}\\
&&& \frac{1}{\sqrt{7}}\\
\end{bmatrix}_{3\times 3}\\
\end{bmatrix}_{7\times 7}.
$$
By performing ${\mathbf{D}}({\mathbb{M},\mathbf{y}})$
on $\widehat{\mathbf{y}}'$ we get a locally-scaled version of
$\widehat{\mathbf{y}}'$ which yields the following
signals $\widehat{\mathbf{y}}''$ and $\widehat{\mathbf{y}}_{\sigma}$ given by
\begin{align*}
\widehat{\mathbf{y}}''=(\frac{1}{\sqrt{21}},\frac{1}{\sqrt{7}},0,0,\frac{1}{\sqrt{7}},0,0)
=( 0.2182, 0.3780,
 0, 0, 0.3780, 0, 0 ).
\end{align*}
and
\begin{align*}
\widehat{\mathbf{y}}_{\sigma}=(\frac{1}{\sqrt{21}},\frac{1}{\sqrt{7}},0,\frac{1}{\sqrt{7}},0,0,0)
=( 0.2182, 0.3780,
 0, 0.3780, 0, 0,0 ).
\end{align*}
Again using IDFT for $\widehat{\mathbf{y}}_{\sigma}$, we get
\begin{align*}
&{\mathbf{y}}_{\sigma}=( 0.3682, 0.0428 + 0.1737i, 0.1398 + 0.0276i, -0.0780 + 0.2013i,
\\&  -0.0780 - 0.2013i, 0.1398 - 0.0276i,  0.0428 - 0.1737i ).
 \end{align*}
 Also
\begin{align*}
H_0=\mathbb{M}=\{1,2,4\},\\
H_1=\{3,6,5\}.
\end{align*}
By Theorem \ref{f.f.s.coset} and a simple calculation for any $\mathbf{x}\in \mathbb{C}^7$, we have
\begin{align*}
&\sum_{m\in\mathbb{M}}\sum_{k=0}^{6}|\langle\mathbf{x},T_kD_m\mathbf{y_{\sigma}}\rangle|^2
=7\bigg(3 |\widehat{\mathbf{x}}(0)|^2 |\widehat{\mathbf{y}}_{\sigma}(0)|^2 + \sum_{t=0}^{1}\bigg(\sum_{\ell\in H_t}|\widehat{\mathbf{x}}(\ell)|^2\bigg)\bigg(\sum_{\ell\in H_t}|\widehat{\mathbf{y}}_{\sigma}(\ell)|^2\bigg)\bigg)
\\&=7\bigg(3 |\widehat{\mathbf{x}}(0)|^2 |\widehat{\mathbf{y}}_{\sigma}(0)|^2 + \bigg(|\widehat{\mathbf{x}}(1)|^2+|\widehat{\mathbf{x}}(2)|^2 + |\widehat{\mathbf{x}}(4)|^2\bigg)\bigg(|\widehat{\mathbf{y}}_{\sigma}(1)|^2+|\widehat{\mathbf{y}}_{\sigma}(2)|^2 + |\widehat{\mathbf{y}}_{\sigma}(4)|^2\bigg)
\\&+\bigg(|\widehat{\mathbf{x}}(3)|^2+|\widehat{\mathbf{x}}(6)|^2 + |\widehat{\mathbf{x}}(5)|^2\bigg)\bigg(|\widehat{\mathbf{y}}_{\sigma}(3)|^2+|\widehat{\mathbf{y}}_{\sigma}(6)|^2 + |\widehat{\mathbf{y}}_{\sigma}(5)|^2\bigg)\bigg)
\\&=7\bigg(3 |\widehat{\mathbf{x}}(0)|^2 |\frac{1}{\sqrt{21}}|^2 + \bigg(|\widehat{\mathbf{x}}(1)|^2+|\widehat{\mathbf{x}}(2)|^2 + |\widehat{\mathbf{x}}(4)|^2\bigg)\bigg(|\frac{1}{\sqrt{7}}|^2+|0|^2 + |0|^2\bigg)
\\&+\bigg(|\widehat{\mathbf{x}}(3)|^2+|\widehat{\mathbf{x}}(6)|^2 + |\widehat{\mathbf{x}}(5)|^2\bigg)\bigg(|\frac{1}{\sqrt{7}}|^2+|0|^2 + |0|^2\bigg)\bigg)
=\|\mathbf{x}\|^2
\end{align*}
Hence, $\mathcal{W}(\mathbf{y_{\sigma}},\Delta_\mathbb{M})$ is an equal norm Parseval finite wavelet frame for $\mathbb{C}^7$.
\end{example}
%%%%%%%%%%%%%%%%%%%%%%%%%%%%%%%%%%%%%%%%%%%%%%%%%%%%%%%%%
In \cite{RS.AC}, the matrix notion presented in Corollary \ref{matrix.int}, have been applied
as a useful tool to determine whether a finite wavelet system
forms a frame for $\mathbb{C}^p$.
As another application of this notion, the next result for any given window function derives a characterization of all multiplicative subgrups of $\mathbb{U}_p$,
for which the associated wavelet system
form
 frames for $\mathbb{C}^p$.
%%%%%%%%%%%%%%%%%%%%%%%%%%%%%%%%%%%%%%%%%%
\begin{theorem}\label{car}
{\it Let $p$ be a positive prime integer, $\mathbf{y}\in \mathbb{C}^p$ such that $\widehat{\mathbf{y}}(0)\neq 0$, and let
\begin{align*}
\Phi(p)=p-1= \prod_{i=1}^{k} q_i^{{\alpha}_i}
\end{align*}
be the  factorization of $\Phi(p)$ into prime powers, where the prime factors $q_i$ are distinct such that
$q_1< q_2<...<q_k$, $k\geq1$ and ${{\alpha}_i}\geq 0$. Let $\epsilon$ be a generator of $\mathbb{U}_p$ and
\begin{align*}
\Lambda:=& \{(r_1,r_2,...,r_k) : \prod_{i=1}^{k}q_i^{r_i}\leq \|\widehat{\mathbf{y}}\|_{0}-1,  0\leq r_i\leq \alpha_{i},\hspace{.1cm}
{\mathbf{Y}}(<\epsilon^{\prod_{i=1}^{k}q_i^{r_i}}>,\mathbf{y})\hspace{.1cm}
\\&has\hspace{.1cm} \prod_{i=1}^{k}q_i^{r_i} nonzero\hspace{.1cm}rows\}.
\end{align*}
Then the set of all subgroups of $\mathbb{U}_p$ for which the associated finite wavelet systems
are frames for $\mathbb{C}^p$ consists of exactly those subgroups $\mathbb{M}$ of
$\mathbb{U}_p$ of the form
\begin{align*}
\mathbb{M}=<\epsilon^{\prod_{i=1}^{k}q_i^{r_i-s_i}}>,\hspace{.5cm}((r_1,...,r_k)\in\Lambda,\hspace{.1cm}0\leq s_i\leq r_i,\hspace{.1cm} i=1,...,k).
\end{align*}
Hence, the order of any such $\mathbb{M}$ is
\begin{align*}
|\mathbb{M}|=\prod_{i=1}^{k}q_i^{\alpha_i-r_i+s_i}
\end{align*}
}
\end{theorem}
%__________________________________________________________________________________________________
\begin{proof}
Let $m:=\|\widehat{\mathbf{y}}\|_0-1$. Since $p-1=\Phi(p)=|\mathbb{U}_p|$, so the set of divisors
of $|\mathbb{U}_p|$ is the set of numbers of the form $\prod_{i=1}^{k}q_i^{r_i}$, where
$0\leq r_i\leq \alpha_i$ for $i=1,...,k$.

Suppose $\mathbb{M}$ be an arbitrary multiplicative subgroup of $\mathbb{U}_p$.
Then, we have $\mathbb{M}=<\epsilon^{\frac{p-1}{(|\mathbb{M}|,p-1)}}>=<\epsilon^{\frac{p-1}{|\mathbb{M}|}}>$.
Let  $a_{\mathbb{M}}:=\frac{p-1}{|\mathbb{M}|}$.
The order of $\mathbb{M}$  is a divisor of $p-1$, thus $a_{\mathbb{M}}$ is of the form $\prod_{i=1}^{k}q_i^{r_i}$
for some $0\leq r_i\leq \alpha_i$, $i=1,...,k$.
For any such $\mathbb{M}$, there are two cases to consider, $a_{\mathbb{M}}> m$ or $a_{\mathbb{M}}\leq m$.

Case $a_{\mathbb{M}}=\prod_{i=1}^{k}q_i^{r_i}> m$. In this case, by Corollary \ref{matrix.int},
the frame conditions for the finite wavelet system $\mathcal{W}(\mathbf{y},\Delta_\mathbb{M})$
can not be satisfied. In fact, we have some zero rows in ${\mathbf{Y}}(\mathbb{M},\mathbf{y})$.

Case $a_{\mathbb{M}}=\prod_{i=1}^{k}q_i^{r_i}\leq m$. Again applying Corollary \ref{matrix.int},
 $\mathcal{W}(\mathbf{y},\Delta_\mathbb{M})$ is a finite wavelet frame if and only if
${\mathbf{Y}}(\mathbb{M},\mathbf{y})$ has $a_{\mathbb{M}}$ non-zero rows.
Also if $\mathbb{M}'$ be a subgroup of $\mathbb{U}_p$ such that $|\mathbb{M}|=\prod_{i=1}^{k}q_i^{\alpha_i -r_i} \bigg||\mathbb{M}'|$, i.e., $\mathbb{M}$ is a subgroup of $\mathbb{M}'$ then $\mathcal{W}(\mathbf{y},\Delta_\mathbb{M})
\subseteq \mathcal{W}(\mathbf{y},\Delta_\mathbb{M'})$.
By this, if $\mathcal{W}(\mathbf{y},\Delta_\mathbb{M})$ is a frame for $\mathbb{C}^p$ then, $\mathcal{W}(\mathbf{y},\Delta_\mathbb{M'})$ is also a frame for $\mathbb{C}^p$.
In other words, the remaining cases of subgroups $\mathbb{M'}$ of $\mathbb{U}_p$ satisfying the frame conditions  of the
finite wavelet system $\mathcal{W}(\mathbf{y},\Delta_\mathbb{M'})$ for $\mathbb{C}^p$ are those of the form,
$\mathbb{M}=<\epsilon^{\prod_{i=1}^{k}q_i^{r_i-s_i}}>$
or
 $|\mathbb{M}'|=\prod_{i=1}^{k}q_i^{\alpha_i-r_i+s_i}$ for $(r_1,...,r_k)\in\Lambda$
and $0\leq s_i\leq r_i $.
%Notice that taking $s_i=0$ for $i=0,...,k$, any $(r_i,...,r_k)\in \Lambda$ lies also in this set.
\end{proof}
%%%%%%%%%%%%%%%%%%%%%%%%%%%%%%%%%%%%%%%%%%%%%%%%
Next example gives a numerical illustration of Theorem \ref{car}.
%%%%%%%%%%%%%%%%%%%%%%%%%%%%%%%%%%%%%%%%%%%%%%
\begin{example}
Let $p=13$ and $\mathbf{y}\in \mathbb{C}^{13}$, such that
\begin{align*}
\widehat{\mathbf{y}}=a\delta_0+b\delta_2+c\delta_3+d\delta_{8}+e\delta_{11}+f\delta_{12}, \ \   \ \  \ \ (a,b,c,d,e,f \neq0),
\end{align*}
where $\{\delta_i\}_{i=0}^{12}$
is an orthonormal basis for $\mathbb{C}^{13}$. Then $\|\widehat{\mathbf{y}}\|_0 -1=5$ and
$|\mathbb{U}_{13}|=\Phi(13)=12=2^2 \times 3^1$. Also $\epsilon=2$ is a generator of
$\mathbb{U}_{13}$. The only divisors of $12$ so that are equal or less than $5$ are, 1,2,3,4. Thus by definition of $\Lambda$, we have $\Lambda\subseteq\{(0,0),(1,0),(0,1),(2,0)\}$.

Let $m_1=(2,0)$, then
 $\mathbb{M}_1=<2^4>=<3>$ and we get
$$
{\mathbf{Y}}({\mathbb{M}_1,\mathbf{y}}):=
\begin{bmatrix}
 \widehat{\mathbf {y}}(1)& \widehat{\mathbf {y}}(3)& \widehat{\mathbf {y}}(9) \\
 \widehat{\mathbf {y}}(2)& \widehat{\mathbf {y}}(6)& \widehat{\mathbf {y}}(5)   \\
\widehat{\mathbf {y}}(4)& \widehat{\mathbf {y}}(12)& \widehat{\mathbf {y}}(10) \\
\widehat{\mathbf {y}}(8)& \widehat{\mathbf {y}}(11)& \widehat{\mathbf {y}}(7)
 \end{bmatrix}_{4\times 3}
 =\begin{bmatrix}
 0& c& 0 \\
 b& 0&0 \\
0&f& 0\\
 d&e& 0
 \end{bmatrix}_{4\times 3}
 $$
Thus, $\mathcal{W}(\mathbf{y},\Delta_{\mathbb{M}_1})$ is a frame for $\mathbb{C}^{13}$.
Since $\mathbb{M}_1$ is of order 3, hence the corresponding wavelet systems to subgroups of sizes
6 and 12  i.e., $\mathbb{M}_2=<2^2>=<4>$ and $\mathbb{M}_3=<2^1>=\mathbb{U}_{13}$ are also frames for $\mathbb{C}^{13}$. Note that $\mathbb{M}_2$ and $\mathbb{M}_3$ are corresponding subgroups  to $m_2=(1,0)$ and $m_3=(0,0)$ respectively.
So, now we just examine in the case of $m_4=(0,1)$. We have $\mathbb{M}_4=<2^3>=<8>$, and
$$
{\mathbf{Y}}({\mathbb{M}_4,\mathbf{y}}):=
\begin{bmatrix}
 \widehat{\mathbf {y}}(1)& \widehat{\mathbf {y}}(8)& \widehat{\mathbf {y}}(12) & \widehat{\mathbf {y}}(5) \\
 \widehat{\mathbf {y}}(2)& \widehat{\mathbf {y}}(3)& \widehat{\mathbf {y}}(11) & \widehat{\mathbf {y}}(10)  \\
\widehat{\mathbf {y}}(4)& \widehat{\mathbf {y}}(6)& \widehat{\mathbf {y}}(9) & \widehat{\mathbf {y}}(7)
 \end{bmatrix}_{3\times 4}
 =\begin{bmatrix}
 0& d& f & 0 \\
 b& c&0 & e \\
0&0& 0 & 0
 \end{bmatrix}_{3\times 4}
 $$
Hence, $\mathcal{W}(\mathbf{y},\Delta_{\mathbb{M}_4})$ is not a frame for $\mathbb{C}^{13}$.
\end{example}
}
%%%%%%%%%%%%%%%%%%%%%%%%%%%%%%%%%%%%%%%%%%%%%%%%%%%%%%%%
\noindent\textbf{Acknowledgements}

\vspace{.5cm}
Some of the results are obtained during  the second author's appointment from the NuHAG group at the University of Vienna. We would like to thank Prof. Hans. G. Feichtinger for his valuable comments and the group for their hospitality.

\hspace{1in}
%%%%%%%%%%%%%%%%%%%%%%%%%%%%%%%%%%%%%%%%%%%%%%%%%%%%%%%%%

\end{document}